\documentclass[11pt,reqno, final]{amsart}

\usepackage{color}
\usepackage[colorlinks=true, allcolors=blue]{hyperref}
\usepackage{amsmath, amssymb, amsthm}
\usepackage{mathrsfs}
\usepackage{mathtools}
\usepackage[noabbrev,capitalize,nameinlink]{cleveref}
\crefname{equation}{}{}
\usepackage{fullpage}
\usepackage[noadjust]{cite}
\usepackage{graphics}
\usepackage{pifont}
\usepackage{tikz}
\usepackage{bbm}
\usepackage[T1]{fontenc}

\usetikzlibrary{arrows.meta}

\usepackage{environ}
\usepackage{framed}
\usepackage{url}
\usepackage[linesnumbered,ruled,vlined]{algorithm2e}
\usepackage[noend]{algpseudocode}
\usepackage[labelfont=bf]{caption}
\usepackage{cite}
\usepackage{framed}
\usepackage[framemethod=tikz]{mdframed}
\usepackage{appendix}
\usepackage{graphicx}
\usepackage[textsize=tiny]{todonotes}
\usepackage{tcolorbox}
\usepackage{enumerate}
\allowdisplaybreaks[1]
\usepackage{enumerate}

\crefname{algocf}{Algorithm}{Algorithms}

\crefname{equation}{}{} 
\AtBeginEnvironment{appendices}{\crefalias{section}{appendix}} 

\usepackage[color]{showkeys} 

\colorlet{refkey}{orange!20}
\colorlet{labelkey}{blue!30}

\crefname{algocf}{Algorithm}{Algorithms}

\numberwithin{equation}{section}
\newtheorem{theorem}{Theorem}[section]
\newtheorem{proposition}[theorem]{Proposition}
\newtheorem{lemma}[theorem]{Lemma}

\crefname{claim}{Claim}{Claims}

\newtheorem*{question*}{Question}

\theoremstyle{definition}

\newtheorem*{definition*}{Definition}

\theoremstyle{remark}
\newtheorem*{remark}{Remark}

\newcommand{\snorm}[1]{\lVert#1\rVert}

\newcommand{\mb}{\mathbb}
\newcommand{\mbf}{\mathbf}
\newcommand{\mbm}{\mathbbm}
\newcommand{\mc}{\mathcal}

\newcommand{\on}{\operatorname}
\newcommand{\wh}{\widehat}
\newcommand{\wt}{\widetilde}

\allowdisplaybreaks
\title{Optimal and algorithmic norm regularization of random matrices}

\author[A1]{Vishesh Jain}
\address{Simons Institute for the Theory of Computing, Berkeley, CA 94720, USA}
\email{visheshj@stanford.edu}

\author[A2]{Ashwin Sah}
\author[A3]{Mehtaab Sawhney}
\address{Department of Mathematics, Massachusetts Institute of Technology, Cambridge, MA 02139, USA}
\email{\{asah,msawhney\}@mit.edu}

\begin{document}

\begin{abstract}
Let $A$ be an $n\times n$ random matrix whose entries are i.i.d.~with mean $0$ and variance $1$. We present a deterministic polynomial time algorithm which, with probability at least $1-2\exp(-\Omega(\epsilon n))$ in the choice of $A$, finds an $\epsilon n \times \epsilon n$ sub-matrix such that zeroing it out results in $\wt{A}$ with
\[\snorm{\wt{A}} = O(\sqrt{n/\epsilon}).\]
Our result is optimal up to a constant factor and improves previous results of Rebrova and Vershynin, and Rebrova. We also prove an analogous result for $A$ a symmetric $n\times n$ random matrix whose upper-diagonal entries are i.i.d.~with mean $0$ and variance $1$. 
\end{abstract}

\maketitle

\section{Introduction}\label{sec:introduction}
Recall that the operator norm of an $n\times n$ real-valued matrix $A$ is defined as
\[\snorm{A} := \sup_{x \in \mb{S}^{n-1}}\snorm{Ax}_2,\]
where $\snorm{\cdot}_{2}$ denotes the Euclidean norm and $\mb{S}^{n-1}$ denotes the unit sphere in $\mb{R}^{n}$. The operator norm is a fundamental quantity of interest in the non-asymptotic theory of random matrices (see, e.g., \cite{RV10} and the references therein). A classical result of Bai, Krishnaiah, and Yin \cite{BSY88} shows that if the entries of an $n\times n$ random matrix $A$ are i.i.d.~random variables with $0$ mean, unit variance, and bounded fourth moment, then
\[\snorm{A} = (2 + o(1))\sqrt{n}.\]
The finite fourth moment hypothesis is sharp in the sense that for a sequence of $n\times n$ random matrices $A_n$ with entries that are i.i.d.~random variables with mean $0$, unit variance, and infinite fourth moment, Bai, Silverstein, and Yin \cite{BSY88} showed that 
\[\limsup \frac{\snorm{A}}{\sqrt{n}} = \infty \text{ almost surely.}\]

Motivated by works of Feige and Ofek \cite{FO05} and Le, Levina, and Vershynin \cite{LLV17} on the regularization of the norm of adjacency matrices of random graphs, Rebrova and Vershynin \cite{RV18} asked whether enforcing the bound $\snorm{A} = O(\sqrt{n})$ is a ``local problem'' or a ``global problem''. Specifically, they considered an $n\times n$ random matrix $A$ with i.i.d.~entries and asked what assumptions (if any) on the distribution of the entries guarantees that, with high probability, $\snorm{\wt{A}} = O(\sqrt{n})$ for some matrix $\wt{A}$ obtained by modifying $A$ on a small sub-matrix. They showed \cite[Theorem~1.3]{RV18} that this is not possible if the distribution has either non-zero mean or infinite variance; in other words, in this case, there is a ``global problem''. On the other hand, they showed that if the distribution has zero mean and bounded variance, then the problem is ``local''.    
\begin{theorem}[{\cite[Theorem~1.1]{RV18}}]\label{thm:iid-norm-regularization}
Consider an $n\times n$ random matrix $A$ with i.i.d.~entries that have zero mean and unit variance. There exist absolute constants $C_{\ref{thm:iid-norm-regularization}}, c_{\ref{thm:iid-norm-regularization}} > 0$ such that for any $\epsilon\in(0,1/2]$, with probability at least $1-2\exp(-c_{\ref{thm:iid-norm-regularization}}\epsilon n)$, there exists an $\epsilon n \times \epsilon n$ sub-matrix of $A$ such that replacing all of its entries with zero gives a matrix $\wt{A}$ with
\[\snorm{\wt{A}}\le C_{\ref{thm:iid-norm-regularization}}\frac{\log\epsilon^{-1}}{\sqrt{\epsilon}}\sqrt{n}.\]
\end{theorem}

The work of Rebrova and Vershynin \cite{RV18} leaves open several natural questions. 

\subsection{Optimal norm regularization}
Let $\epsilon \in (0,1/10)$. It is easily seen \cite[Remark~1.2]{RV18} that for the $n\times n$ random matrix $A$ whose entries are i.i.d.~random variables taking the values $0$ with probability $1-(2\epsilon/n)$ and $\pm \sqrt{n/(2\epsilon)}$ with probability $\epsilon/n$ each, with probability at least $1-2\exp(-c'\epsilon n)$,
\[\snorm{\wt{A}} \gtrsim\sqrt{\frac{n}{\epsilon}}\]
for every matrix $\wt{A}$ obtained by modifying $A$ on a $c\epsilon n \times c\epsilon n$ matrix. Here, $c,c' > 0$ are absolute constants. This example shows that the dependence of the bound on $\snorm{\wt{A}}$ in \cref{thm:iid-norm-regularization} is optimal up to a possible factor of $\log\epsilon^{-1}$. In \cite[Section~11]{RV18}, Rebrova and Vershynin asked whether this factor of $\log\epsilon^{-1}$ is necessary. We show that it is not, thereby obtaining a result which is optimal up to constants.    
\begin{theorem}\label{thm:optimal-iid-norm-regularization}
Consider an $n\times n$ random matrix $A$ with i.i.d.~entries that have zero mean and unit variance. There exist absolute constants $C_{\ref{thm:optimal-iid-norm-regularization}}, c_{\ref{thm:optimal-iid-norm-regularization}} > 0$ such that for any $\epsilon\in(0,1/2]$, with probability at least $1-2\exp(-c_{\ref{thm:optimal-iid-norm-regularization}}\epsilon n)$, there exists an $\epsilon n \times \epsilon n$ sub-matrix of $A$ such that replacing all of its entries with zero gives a matrix $\wt{A}$ with 
\[\snorm{\wt{A}}\le C_{\ref{thm:optimal-iid-norm-regularization}}\sqrt{\frac{n}{\epsilon}}.\]
\end{theorem}

\subsection{Norm regularization of random symmetric matrices} The regularization results in \cite{FO05,LLV17} were proved for adjacency matrices of random graphs, whereas the main result of \cite{RV18} holds only for random matrices with i.i.d.~entries. Answering a question in  \cite[Section~11]{RV18}, our next result provides a symmetric counterpart of \cref{thm:optimal-iid-norm-regularization}.  

\begin{theorem}\label{thm:sym-norm-regularization}
Consider an $n\times n$ random symmetric matrix $A$ with i.i.d.~entries on and above the diagonal that have zero mean and unit variance. There exist absolute constants $C_{\ref{thm:sym-norm-regularization}}, c_{\ref{thm:sym-norm-regularization}} > 0 $ such that for any $\epsilon\in(0,1/2]$, with probability at least $1-2\exp(-c_{\ref{thm:sym-norm-regularization}}\epsilon n)$, there exists an $\epsilon n \times \epsilon n$ sub-matrix of $A$ such that replacing all of its entries with zero gives a matrix $\wt{A}$ with 
\[\snorm{\wt{A}}\le C_{\ref{thm:sym-norm-regularization}}\sqrt{\frac{n}{\epsilon}}.\]
\end{theorem}
\begin{remark}
The symmetric version of the example given above shows that the dependence of the bound on $\snorm{\wt{A}}$ is optimal in terms of $\epsilon$. 
Also, one can allow for the diagonals to be arbitrary independent random variables with zero mean and unit variance; this is a straightforward modification of the proof and we leave the details to the interested reader.
\end{remark}

\subsection{Constructive norm regularization}
The norm regularization result of Rebrova and Vershynin (\cref{thm:iid-norm-regularization}) is only an existential result and does not provide a way to efficiently find an appropriate $\epsilon n \times \epsilon n$ sub-matrix to zero out. In contrast, the regularization procedures of \cite{FO05, LLV17} are algorithmic in nature. In \cite[Section~11]{RV18}, Rebrova and Vershynin asked whether one can obtain an explicit description of an $\epsilon n \times \epsilon n$ matrix whose removal regularizes the norm. 

This question was the focus of the work of Rebrova \cite{Reb20} who showed \cite[Corollary~1.3]{Reb20} that for an $n\times n$ random matrix $A$ with i.i.d.~entries having a symmetric distribution and unit variance, for any $\epsilon \in (0,1/2]$, and for any $r \ge 1$, there is a deterministic, polynomial time algorithm to zero out an $\epsilon n \times \epsilon n$ sub-matrix in order to obtain $\wt{A}$ satisfying 
\[\snorm{\wt{A}} \lesssim r^{3/2}\cdot \sqrt{\log \log n \cdot \log \epsilon^{-1}} \cdot \sqrt{\frac{n}{\epsilon}}\]
such that the algorithm succeeds with probability $1 - n^{0.1 - r}$ (in the choice of $A$). Compared to the existential regularization results for i.i.d.~matrices discussed earlier, this result requires that the common distribution of the entries is symmetric (as opposed to only mean $0$), loses an additional factor of $\sqrt{\log \log n \cdot \log \epsilon^{-1}}$ in the bound on $\snorm{\wt{A}}$, and moreover, the failure probability (over the choice of $A$) of the regularization procedure is much larger than for the existential results. 

Our final result remedies these shortcomings, thereby providing constructive versions of \cref{thm:optimal-iid-norm-regularization,thm:sym-norm-regularization}

\begin{theorem}\label{thm:algorithmic-norm-regularization}
The $\epsilon n\times\epsilon n$ matrices guaranteed in \cref{thm:optimal-iid-norm-regularization,thm:sym-norm-regularization} can be found via a deterministic polynomial time algorithm which is guaranteed to succeed with probability at least $1-2\exp(-c_{\ref{thm:algorithmic-norm-regularization}}\epsilon n)$ \emph{(}in the choice of $A$\emph{)}. Here, $c_{\ref{thm:algorithmic-norm-regularization}} > 0$ is an absolute constant.  
\end{theorem}




\subsection{Organization}
As in \cite{RV18}, we follow the natural high-level strategy of decomposing $A$ into a number of parts based on the magnitude of the entries as well as controlling the operator norm of the part with the smallest entries using a version of the Grothendieck-Pietsch factorization theorem. \cref{sec:operator-norm} contains our treatment of ``medium entries'' and ``large entries'' while \cref{sec:2->infty} and \cref{sub:infty->2} contains our treatment of ``small entries''. Finally, these results are combined in \cref{sec:proofs} to prove \cref{thm:optimal-iid-norm-regularization,thm:sym-norm-regularization,thm:algorithmic-norm-regularization}.

\subsection{Notation}\label{sub:notation}
Given an $n\times n$ matrix $A$ and a subset $J\subseteq[n]$, we let $A_J$ denote the matrix obtained from $A$ by zeroing out the columns $J^{c}$. Given an $n\times n$ matrix $A$ and a subset $S\subseteq[n]\times[n]$, we denote by $A_S$ the matrix obtained from $A$ by zeroing out all entries not in $S$. 

For the sake of uniformity in our arguments, we will deduce all statements for $n\times n$ random matrices with i.i.d.~entries above the diagonal and $0$ on and below the diagonal. For brevity, we will refer to such matrices as i.i.d.~random upper triangular matrices and will typically denote them by $T$. 

\subsection{Acknowledgements}
We thank Liza Rebrova for useful discussions.

\section{Controlling \texorpdfstring{$2\to\infty$}{2 to infinity} norm}\label{sec:2->infty}
As in \cite{RV18}, we begin by regularizing $\snorm{\cdot}_{2 \to \infty}$ for the ``small'' part of the matrix. 

\begin{proposition}\label{lem:sym-2->infty}
Let $\epsilon\in(0,1/2)$. Consider an $n\times n$ upper triangular matrix $T$ with i.i.d.~entries satisfying $\mb{E}T_{ij}^2\le 1$ and $|T_{ij}|\le\sqrt{n/\log\epsilon^{-1}}$ almost surely. There exist absolute constants $C_{\ref{lem:sym-2->infty}}, c_{\ref{lem:sym-2->infty}} > 0$ such that with probability at least $1-2\exp(-c_{\ref{lem:sym-2->infty}}\epsilon n)$, there is a subset $J\subseteq[n]$ with $|J|\le C_{\ref{lem:sym-2->infty}}\epsilon n$ for which
\[\snorm{T_{J^c}}_{2\to\infty}\le C_{\ref{lem:sym-2->infty}}\sqrt{n}.\]
Furthermore, this set can be found algorithmically without access to the underlying distribution.
\end{proposition}

Compared to \cite[Lemma~5.1]{RV18}, the above proposition is valid for i.i.d.~upper triangular matrices (as opposed to only i.i.d.~matrices) and does not lose a factor of $\sqrt{\log\epsilon^{-1}}$ in the bound on $\snorm{\cdot}_{2 \to \infty}$, although this comes at the cost of restricting the bound on $|T_{ij}|$ by a corresponding factor of $1/\sqrt{\log{\epsilon^{-1}}}$.  

The key step in the proof of \cref{lem:sym-2->infty} is the following, which is closely related to \cite[Theorem~4.2]{RV18}. 
\begin{lemma}\label{prop:damping-sum}
Let $\epsilon\in(0,1/2)$ and let $\xi \in [0,n/\log\epsilon^{-1}]$ be such that $\mb{E}\xi\le 1$. Let $X_1,\ldots,X_n$ as well as $Y_1,\ldots,Y_{n^2/8}$ be independent samples of $\xi$. There exist absolute constants $C_{\ref{prop:damping-sum}}, c_{\ref{prop:damping-sum}} > 0$ for which the following holds. For all $n\ge C_{\ref{prop:damping-sum}}\log\epsilon^{-1}$, there exist random variables $W_1,\ldots,W_n\in[0,1]$ depending only on $X_1,\ldots,X_n,Y_1,\ldots,Y_{n^2/8}$ (and not the distribution of $\xi$) such that with probability at least $1-2\exp(-c_{\ref{prop:damping-sum}}n\log\epsilon^{-1})$ over $Y_1,\ldots,Y_{n^2/8}$, we have
\[\quad \quad \quad \quad \quad \quad \quad \quad \quad \quad  \sum_{j=1}^nW_jX_j\le C_{\ref{prop:damping-sum}}n \quad \text{ almost surely over }X_1,\dots, X_n, \text { and}\]
\[1\le\mb{E}_{X_1,\ldots,X_n}\bigg(\prod_{i=1}^nW_i\bigg)^{-1}\le 1+\epsilon.\]
\end{lemma}
\begin{proof}
Let $K = 1/\log\epsilon^{-1}$ so that $\xi \in [0,Kn]$ and $n = \Omega(1/K)$. Let $t = \lceil \log_{2}(Kn)\rceil$. Let $q_{-1} = q_0 = 0, q_{t+1} = Kn$, and for $k \in [t]$, let
\[q_k = \sup\{r \ge 0: \mb{P}[\xi \ge r] \ge 2^{-k}\}.\]
Let $\wh{\mb{P}}$ denote the empirical measure generated by $Y_1,\dots, Y_{n^{2}/8}$. Let $\wh{q}_0 = 0, \wh{q}_{t+1} = 2Kn$, and for $k \in [t]$, let
\[\wh{q}_k = \sup\{r \ge 0: \wh{\mb{P}}\left[[r, \infty)\right] \ge 2^{-k}\}.\]
Let $\mc{G}$ denote the event, measurable with respect to $Y_1,\dots, Y_{n^2/8}$, that
\[q_{k-1} \le \wh{q}_k \le q_{k+1}\quad \forall k \in [t].\]
Then, it follows from a staightforward application of the Chernoff bound and the union bound that
\[\mb{P}_{Y_1,\dots, Y_{n^2/8}}[\mc{G}] \ge 1-\sum_{k=1}^{t}\exp\left(-\Omega\left(\frac{n^2}{2^k}\right)\right) \ge 1-2\exp(-\Omega(n/K)).\]

Now, fix a realisation of $Y_1,\dots, Y_{n^2/8}$ and the associated quantities $\wh{q}_0,\dots,\wh{q}_{t+1}$. Let $k_0 = 0$ and inductively define
\[k_{i+1} = \max\{k_{i} + 1, \max\{j \in (k_{i}, {t+1}] : \wh{q}_j < 2 \wh{q}_{k_{i}}\}\}.\]
Here, the maximum of the empty set is $-\infty$. Let $\ell$ denote the first index for which $k_\ell = t+1$. Note that for all $i \in \{0,1,\dots, \ell-1\}$,
\[2^{-k_i} \wh{q}_{k_{i+1}} \lesssim 2^{-k_i}\wh{q}_{k_i} + 2^{-k_{i+1}}\wh{q}_{k_{i+1}}.\]

Let $L \ge 1$ be a sufficiently large absolute constant to be specified later. For a realisation of $X_1,\dots, X_n$ and for $i \in \{0,1,\dots, \ell-1\}$, let
\[\nu_i = \#\{j \in [n]: X_j\in[\wh{q}_{k_i},\wh{q}_{k_{i+1}})\} \text{ and } w_i = \min\left(1, \frac{L2^{-k_{i}}n}{\nu_i}\right).\]
Also, for $j \in [n]$, we let $i(j)$ denote the unique (by construction) index $i \in \{0,1,\dots\ell-1\}$ for which $X_j \in [\wh{q}_{k_i}, \wh{q}_{k_{i+1}})$. For $j \in [n]$ and $i\in \{0,1,\dots,\ell-1\}$, let
\[W_j = w_{i(j)} \text{ and } Z_i = \prod_{j: i(j) = i}W_j^{-1}.\]
Then, on the event $\mc{G}$, we have
\begin{align*}
    \sum_{j=1}^{n}W_j X_j &= \sum_{i=0}^{\ell-1} \sum_{j: i(j) = i}W_j X_j \le \sum_{i=0}^{\ell-1}Ln\cdot 2^{-k_i}\wh{q}_{k_{i+1}} \lesssim \sum_{k=0}^{t+1}Ln\cdot 2^{-k}\wh{q}_{k} \\
    &\lesssim Ln + \sum_{k=1}^{t} L n \cdot 2^{-k}q_{k+1}\lesssim Ln\left(1 + \mb{E}\xi\right) \lesssim Ln. 
\end{align*}
Moreover, on the event $\mc{G}$, we have for any $i \in \{0,1,\dots, \ell-1\}$ that
\begin{align*}
    \mb{E}_{X_1,\dots, X_n}[Z_i]  
    &\le 1 + \mb{E}\left[\left(\frac{\nu_i}{L2^{-k_{i}}n}\right)^{\nu_i}\mbm{1}_{\nu_i > L2^{-k_i}n}\right]\\
    &\le 1 + \sum_{s > L2^{-k_i}n}\left(\frac{s}{L2^{-k_i}n}\right)^{s}\mb{P}[\nu_i = s]\\
    &\le 1 + \sum_{s > L2^{-k_i}n}\left(\frac{s}{L2^{-k_i}n}\right)^{s}\binom{n}{s}\mb{P}[\xi \ge \wh{q}_{k_{i}}]^{s}\\
    &\le 1 + \sum_{s > L2^{-k_i}n}\left(\frac{e}{L2^{-k_i}}\right)^{s}\mb{P}[\xi \ge q_{k_i - 1}]^{s}\\
    &\le 1 + \sum_{s > L2^{-k_i}n}\left(\frac{4e}{L}\right)^{s} \le \exp\left(2\left(\frac{8e}{L}\right)^{\frac{Ln}{2^{2 + k_i}}}\right),
\end{align*}
for a sufficiently large absolute constant $L$. 
Then, since the random variables $|\nu_0|,\dots, |\nu_{\ell-1}|$ are negatively associated and since $Z_i$ is an increasing function of $|\nu_i|$, it follows that
\begin{align*}
    \mb{E}_{X_1,\dots, X_n}\left[\prod_{j=1}^{n} W_j^{-1}\right] 
    &= \mb{E}\left[\prod_{i=0}^{\ell-1}Z_i\right]
    \le \prod_{i=0}^{\ell-1}\mb{E}[Z_i]\\
    &\le \exp\left(2\left(\frac{8e}{L}\right)^{L/(8K)}\right) \le 1+\epsilon
\end{align*}
for a sufficiently large absolute constant $L$, where the final inequality uses that $K = 1/\log\epsilon^{-1}$.
\end{proof}

We can now quickly deduce \cref{lem:sym-2->infty}.

\begin{proof}[Proof of \cref{lem:sym-2->infty}]
We may assume that $n = \Omega(\epsilon^{-1})$ as otherwise, the desired probability bound is negative. Moreover, by adding an extra row and column of zeros (if necessary), we may assume that $n$ is even. Recall that the diagonal entries of $T$ are $0$. Consider the $n/2 \times n/2$ matrices
\[N_{ij} = T_{i, n/2 + j}, \text{ and }\]
\[N'_{ij} = 
\begin{cases}
T_{i+n/2, j + n/2} &\text { if } j > i \text{ and } i < n/2 \\
T_{n/2 - i, n/2 - j + 1} &\text{ if } j \le i \text{ and } i < n/2\\
0 &\text { if } i = n/2.
\end{cases}
\]
Then, it is straightforward to see that if we can find $O(\epsilon n)$ columns in $N$ and $N'$ to zero-out such that the resulting matrices have $\snorm{\cdot}_{2 \to \infty}$ norm $O(\sqrt{n})$, then the same is true for $T$ and in fact, the choice of columns used for $N$ and $N'$ correspond to an obvious choice of columns for $T$.

By taking the union bound, it suffices to show the following: with probability at least  $1 - 2\exp(-\Omega(n\log{\epsilon^{-1}}))$ over the realisation of $N'$, with probability at least $1 - 2\exp(-\Omega(\epsilon n))$ over the realization of $N$, we can algorithmically find a set of $O(\epsilon n)$ columns of $N$ to zero-out such that the resulting matrix has $\snorm{\cdot}_{2\to \infty}$ norm $O(\sqrt{n})$. This follows from a direct application of \cref{prop:damping-sum}. 

Indeed, we treat the entries $(N'_{ij})^{2}$ for $i < n/2$ as the i.i.d.~samples $Y_1,\dots, Y_{n^2/8}$. Then, by \cref{prop:damping-sum}, with probability at least $1-2\exp(-\Omega(n\log\epsilon^{-1}))$ over the realization of $N'$, the following holds. For any $i \in [n/2]$, we can find $W_{i1},\dots, W_{i,n/2} \in [0,1]$ depending only on $N_{i1},\dots, N_{i,n/2}$ and $N'$ such that 
\[\sum_{j=1}^{n/2} W_{ij}N_{ij}^{2} \le C_{\ref{prop:damping-sum}}n \text{ and }\]
\[1\le\mb{E}_{N_{i1},\ldots,N_{i,n/2}}\bigg(\prod_{j=1}^{n/2}W_{ij}\bigg)^{-1}\le\exp(\epsilon).\]
Let 
\[V_j = \prod_{i=1}^{n/2}W_{ij}, \quad j \in [n/2].\]
Then, $W_{ij} \le V_j$ for all $i \in [n/2]$ so that
\begin{equation}
\label{eqn:damp}
\sum_{j=1}^{n/2} V_j N_{ij}^{2} \le C_{\ref{prop:damping-sum}}n \quad \forall i \in [n/2]\text{ and }
\end{equation}
\begin{equation}
\label{eqn:large-weights}
1\le\mb{E}_N\bigg(\prod_{j=1}^{n/2}V_j\bigg)^{-1}\le\exp(\epsilon n).
\end{equation}
Let $J = \{j \in [n/2]: V_j < e^{-2}\}$. By \cref{eqn:large-weights} and Markov's inequality, it follows that $|J| \le \epsilon n$ with probability at least $1-\exp(-\epsilon n)$ over the choice of $N$. Moreover, by \cref{eqn:damp}, it follows that for all $i \in [n/2]$, $\sum_{j \in J^{c}}N_{ij}^{2} \le e^{2}\cdot C_{\ref{prop:damping-sum}}n$, which completes the proof. 
\end{proof}

\section{Controlling \texorpdfstring{$\infty\to 2$}{infinity to 2} and operator norms}\label{sec:infty->2}
\subsection{\texorpdfstring{$\infty\to 2$}{infinity to 2} norm}\label{sub:infty->2}
In this subsection, we will show how to turn a bound on $\snorm{\cdot}_{2\to\infty}$ into a bound on $\snorm{\cdot}_{\infty\to 2}$. In contrast to the corresponding step in \cite{RV18}, we will be able to accomplish this \emph{without} removing any additional columns, which will be useful for our algorithmic regularization procedure. The goal of this subsection is to prove the following Seginer-type \cite{Seg00} result for the $\snorm{\cdot}_{\infty \to 2}$ norm. 
\begin{proposition}
\label{prop:restrictions}
Let $T$ be an $n\times n$ upper triangular matrix with i.i.d.~entries of mean $0$ and variance at most $1$. There exists an absolute constant $C_{\ref{prop:restrictions}} > 0$ such that with probability at least $1 - 4^{-n}$, the following holds: 
\[\snorm{T_I}_{\infty \to 2} \le C_{\ref{prop:restrictions}}(\sqrt{n}\snorm{T_I}_{2\to \infty} + n) \text{ for all } I \subseteq [n], |I| \ge n/2.\]
\end{proposition}

The proof of \cref{prop:restrictions} will be presented at the end of this subsection following a series of preparatory lemmas. We begin with the following tight relationship between $\snorm{\cdot }_{2\to \infty}, \snorm{\cdot }_{\infty \to 2}$, and $\snorm{\cdot \mbf{1}}_{2}$, proved in \cite{RV18} for a completely i.i.d.~matrix.
\begin{lemma}[{\cite[Lemmas~6.3 and 6.4]{RV18}}]\label{lem:concentration-1}
Let $A$ be a random $n\times n$ matrix with i.i.d.~entries. Then 
\[\mb{E}\snorm{A}_{\infty\to 2}\le C_{\ref{lem:concentration-1}}(\sqrt{n}\mb{E}\snorm{A}_{2\to\infty}+\mb{E}\snorm{A\mbf{1}}_2),\]
and with probability at least $1-e^{-n}$, 
\[\snorm{A}_{\infty\to 2}\le C_{\ref{lem:concentration-1}}(\sqrt{n}\mb{E}\snorm{A}_{2\to\infty}+\mb{E}\snorm{A\mbf{1}}_2).\]
\end{lemma}

We establish an analogous version for random upper triangular matrices with i.i.d.~entries. 

\begin{lemma}\label{lem:concentration-2}
Let $T$ be a random upper triangular $n\times n$ matrix with i.i.d.~entries, and let $J$ be a subset of columns with $|J|\ge n/2$. Then with probability at least $1-2e^{-n}$,
\[\snorm{T_J}_{\infty\to 2}\le C_{\ref{lem:concentration-2}}(\sqrt{n}\mb{E}\snorm{T_J}_{2\to\infty}+\mb{E}\snorm{T_J\mbf{1}}_2).\]
\end{lemma}

During the course of our proof, we will make use of the following estimates regarding random matrix models of different ``shapes''.
\begin{lemma}\label{lem:containment}
Let $R\subseteq S\subseteq [n]\times [n]$. Then
\[\mb{E}\snorm{A_R}_{\infty\to 2}\le 2\mb{E}\snorm{A_S}_{\infty\to 2}\]
and 
\[\mb{E}\snorm{A_R\mbf{1}}_2\le 2\mb{E}\snorm{A_S\mbf{1}}_2.\]
\end{lemma}
\begin{proof}
Let $\mu = \mb{E}\xi$. We have
\begin{align*}
\mb{E}\snorm{A_R}_{\infty\to 2}& = \mb{E}\snorm{A_R + \mb{E}[A_{S\setminus R} - \mb{E}A_{S\setminus R}]}_{\infty \to 2}\\
&\le\mb{E}\snorm{A_R+A_{S\setminus R}-\mb{E}A_{S\setminus R}}_{\infty\to 2}\le\mb{E}\snorm{A_S}_{\infty\to 2}+\snorm{\mu\mbf{1}_{S\setminus R}}_{\infty\to 2}\\
&\le\mb{E}\snorm{A_S}_{\infty\to 2}+\snorm{\mu\mbf{1}_S}_{\infty\to 2} = \mb{E}\snorm{A_S}_{\infty\to 2}+\snorm{\mb{E}A_S}_{\infty\to 2}  \\
&\le 2\mb{E}\snorm{A_S}_{\infty\to 2},
\end{align*}
where we have used Jensen's inequality twice. An analogous proof establishes the second inequality as well. 
\end{proof}

We can now prove \cref{lem:concentration-2}.
\begin{proof}[Proof of \cref{lem:concentration-2}]
We may write $T_J = A_S$ for some  $S\subseteq[n]\times[n]$, where note that $S$ contains an $\lfloor n/4\rfloor\times\lfloor n/4\rfloor$ (not necessarily consecutive) block as a subset. Since zeroing out entries of a matrix cannot increase the $\snorm{\cdot}_{2 \to \infty}$ norm, it follows that 
\[\mb{E}\snorm{A}_{2\to\infty}\lesssim \mb{E}\snorm{A_{\lfloor n/4\rfloor\times\lfloor n/4\rfloor}}_{2\to\infty}\le\mb{E}\snorm{T_J}_{2\to\infty}\le\mb{E}\snorm{A}_{2\to\infty},\]
where the first inequality uses the triangle inequality along with the previously mentioned fact about zeroing out entries. 
Moreover, we have
\[\mb{E}\snorm{A}_{\infty\to 2} \lesssim \mb{E}\snorm{A_{\lfloor n/4\rfloor\times \lfloor n/4\rfloor}}_{\infty\to 2}\lesssim \mb{E}\snorm{T_J}_{\infty\to 2}\lesssim \mb{E}\snorm{A}_{\infty\to 2},\]
where the first inequality uses the triangle inequality and \cref{lem:containment} and the subsequent inequalities use \cref{lem:containment}.
Similarly,
\[\mb{E}\snorm{A\mbf{1}}_2 \lesssim \mb{E}\snorm{A_{\lfloor n/4\rfloor\times\lfloor n/4\rfloor}\mbf{1}}_2\lesssim\mb{E}\snorm{T_J\mbf{1}}_2\lesssim \mb{E}\snorm{A\mbf{1}}_2.\]
Next, write $A = T_J + A_{S'}$ for $S' = ([n]\times[n])\setminus S$. Then,
for $t\ge 2\mb{E}\snorm{A_{S'}}_{\infty\to 2}$, we have
\begin{align*}
\mb{P}[\snorm{A}_{\infty\to 2}\ge t]&\ge\mb{P}[\snorm{A_{S'}}_{\infty\to 2}\le t]\cdot \mb{P}[\snorm{T_J}_{\infty\to 2}\ge 2t]\\
&\ge\mb{P}[\snorm{T_J}_{\infty\to 2}\ge 2t]/2,
\end{align*}
where the second line follows from Markov's inequality.
Since 
\[\mb{E}\snorm{A_{S'}}_{\infty\to 2}\lesssim \mb{E}\snorm{A}_{\infty\to 2}\lesssim\sqrt{n}\mb{E}\snorm{A}_{2\to\infty}+\mb{E}\snorm{A\mbf{1}}_2\lesssim\sqrt{n}\mb{E}\snorm{T_J}_{2\to\infty}+\mb{E}\snorm{T_J\mbf{1}}_2,\]
where the first inequality is by \cref{lem:containment}, the second inequality is by \cref{lem:concentration-1}, and the third is by the previously established inequalities, we may choose 
\[t = C(\sqrt{n}\mb{E}\snorm{T_J}_{2\to\infty}+\mb{E}\snorm{T_J\mbf{1}}_2)\]
with $C$ is a large absolute constant guaranteeing that
\[t\ge (2 + C_{\ref{lem:concentration-1}} )\max(\sqrt{n}\mb{E}\snorm{A}_{2\to\infty}+\mb{E}\snorm{A\mbf{1}}_2,\mb{E}\snorm{A_{S'}}_{\infty\to 2}).\]
Finally, for such a choice of $C$ and $t$, we have by \cref{lem:concentration-1} that
\[\mb{P}[\snorm{T_J}_{\infty\to 2}\ge 2t] \le 2\mb{P}[\snorm{A}_{\infty\to 2}\ge t] \le 2e^{-n},\]
as desired. 
\end{proof}

With the preceding lemma in hand, we can prove the following, which shows that with at least some exponentially \emph{small} probability, the $\snorm{\cdot}_{2\to \infty}$ and $\snorm{\cdot}_{\infty \to 2}$ norms of the matrix are already regularized. 

\begin{lemma}\label{lem:low-prob-2}
Let $T$ be a random upper triangular $n\times n$ matrix with i.i.d.~entries of mean $0$ and variance at most $1$, and let $J$ be a subset of columns with $|J|\ge n/2$. For any $\delta\in(0,1/2)$,
\[\snorm{T_J}_{2\to\infty}\le C_{\ref{lem:low-prob-2}}\delta^{-1}\sqrt{n}\quad\emph{and}\quad\snorm{T_J}_{\infty\to 2}\le C_{\ref{lem:low-prob-2}}\delta^{-1}n\]
with probability at least $\exp(-\delta^2n)/2$.
\end{lemma}
\begin{proof}
The proof is identical to that of \cite[Lemma~6.5]{RV18} with the application of \cite[Lemma~6.4]{RV18} replaced by \cref{lem:concentration-2}.
\end{proof}




Finally, we need the following symmetrization estimate from \cite{RV18}.
\begin{lemma}[{Proof of \cite[Lemma~6.1]{RV18}}]
\label{lem:signs}
Let $A$ be an $n\times n$ matrix and let $\wt{A}$ denote the random matrix with entries $\wt{A}_{ij} = \epsilon_{ij}A_{ij}$, where $\epsilon_{ij}$ are i.i.d.~Rademacher random variables. There exists an absolute constant $c_{\ref{lem:signs}} > 0$ such that for all $t\ge 1$, with probability at least $1 - 2^n\exp(-c_{\ref{lem:signs}}tn)$, 
\[\snorm{\wt{A}}_{\infty \to 2} \le \sqrt{2tn} \snorm{A}_{2\to \infty}.\]
\end{lemma}

We now have all the ingredients needed to prove \cref{prop:restrictions}. 

\begin{proof}[Proof of \cref{prop:restrictions}]
Fix $I \subseteq [n]$ with $|I|\ge n/2$, let $B = T_I$, and let $\wt{B} = B - B'$, where $B'$ denotes an independent copy of $T_I$. \
For $t\ge 1$ and $\delta \in (0,1/2)$, let 
\begin{align*}
    \mc{E}(t) &:= \{\snorm{\wt{B}}_{\infty \to 2} \le \sqrt{2tn}\snorm{\wt{B}}_{2\to \infty}\}, \\
    \mc{F}(\delta) &:= \{\snorm{B'}_{2\to \infty}\le C_{\ref{lem:low-prob-2}}\delta^{-1}\sqrt{n} \wedge \snorm{B'}_{\infty \to 2}\le C_{\ref{lem:low-prob-2}}\delta^{-1}n\}, \quad \text{ and }\\
    \mc{G}(t, \delta)&:= \{\snorm{B}_{\infty \to 2} \le 2\sqrt{2tn}\snorm{B}_{2\to \infty} + (2C_{\ref{lem:low-prob-2}}\delta^{-1}\sqrt{2t} +  C_{\ref{lem:low-prob-2}}\delta^{-1})n \}.
\end{align*}
Then, by the triangle inequality, we have that on the event $\mc{G}(t,\delta)^{c} \cap \mathcal{F}(\delta)$,
\begin{align*}
    \snorm{\wt{B}}_{2\to \infty} &\le \snorm{B}_{2\to \infty} + \snorm{B'}_{2 \to \infty} \le \snorm{B}_{2\to \infty} + C_{\ref{lem:low-prob-2}}\delta^{-1}\sqrt{n}, \quad \text{ and }\\
    \snorm{\wt{B}}_{\infty \to 2} &\ge \snorm{B}_{\infty \to 2} - \snorm{B'}_{\infty \to 2}\ge 2\sqrt{2tn}\bigg(\snorm{B}_{2\to \infty} + C_{\ref{lem:low-prob-2}}\delta^{-1}\sqrt{n}\bigg).
\end{align*}
In particular, this implies
\[\mc{G}(t,\delta)^{c} \cap \mc{F}(\delta) \subseteq \mc{E}^{c}(t).\]
Therefore, since $\mc{G}(t,\delta)$ and $\mc{F}(\delta)$ are independent, we have
\[\mb{P}[\mc{G}(t,\delta)^{c}]\le \frac{\mb{P}[\mc{E}^{c}(t)]}{\mb{P}[\mc{F}(\delta)]}\le \frac{2^n\exp(-c_{\ref{lem:signs}}tn)}{\exp(-\delta^2n)/2}, \]
where we have used \cref{lem:signs,lem:low-prob-2}. Finally, choosing $t = 1000/c_{\ref{lem:signs}}$, $\delta = 1/4$, and taking the union bound over at most $2^n$ choices of $I$, we obtain the desired conclusion.  
\end{proof}

\subsection{Operator norm}\label{sub:operator-norm}
Having thus established control over $\snorm{\cdot}_{\infty \to 2}$, we can apply a version of the Grothendieck-Pietsch theorem \cite[Proposition~15.11]{LT91} as in \cite{LLV17,RV18} to establish control over the operator norm. 
\begin{theorem}\label{thm:grothendieck-pietsch}
Let $B$ be a $k\times m$ real matrix and let $\delta > 0$. There exists $J\subseteq[m]$ with $|J|\le\delta m$ for which
\[\snorm{B_{J^c}}\le\frac{2}{\sqrt{\delta m}}\snorm{B}_{\infty\to 2}.\]
\end{theorem}

Now, applying \cref{lem:sym-2->infty,prop:restrictions,thm:grothendieck-pietsch} in sequence immediately yields the following result.
\begin{proposition}\label{prop:almost-surely-bounded}
Let $\epsilon \in (0,1/2)$ and consider an $n\times n$ upper triangular matrix $T$ with i.i.d.~mean $0$ entries of variance at most $1$, and such that $|T_{ij}|\le\sqrt{n/\log(1/\epsilon)}$ almost surely. Then with probability at least $1-2\exp(-c_{\ref{prop:almost-surely-bounded}}\epsilon n)$, there is a subset $J\subseteq[n]$ with $|J|\le C_{\ref{prop:almost-surely-bounded}}\epsilon n$ such that
\[\snorm{T_{J^c}}\le C_{\ref{prop:almost-surely-bounded}}\sqrt{\frac{n}{\epsilon}}.\]
\end{proposition}

\section{Medium and Large Entries}\label{sec:operator-norm}
In this section, we handle the ``medium'' and ``large'' entries. Let $\epsilon \in (0,1/2)$, and recall that \cref{prop:almost-surely-bounded} takes care of those entries of the matrix $T$ which have absolute value at most $\sqrt{n/\log\epsilon^{-1}}$. We split the remaining entries of $T$ into three separate classes:
\begin{itemize}
\item $|T_{i,j}|\in\left(\sqrt{n/\log\epsilon^{-1}},\sqrt{n/(\epsilon \log^2\epsilon^{-1})}\right]$, \item $|T_{i,j}|\in\left(\sqrt{n/(\epsilon \log^2\epsilon^{-1})},5\sqrt{n/\epsilon}\right]$, and 
\item $|T_{i,j}| > 5\sqrt{n/\epsilon}$.  
\end{itemize}
The first and third classes are handled using arguments similar to \cite{RV18}. The second class requires a more intricate argument, and will be considered at the end of this section.

We will need an elementary lemma from \cite{RV18} which will allow us to combine different sub-matrices that we zero out into one. 
\begin{lemma}[{\cite[Lemma~8.7]{RV18}}]\label{lem:sub-matrix}
Let $M$ be an $n\times n$ matrix. Let $I,J\subseteq [n]$. Then 
\[\snorm{B-B_{I\times J}}\le\snorm{B_{I^c\times[n]}}+\snorm{B_{[n]\times J^c}}\le 2\snorm{B}.\]
\end{lemma}
\begin{proof}
Indeed,
\[\snorm{B-B_{I\times J}}\le\snorm{B_{I^c\times[n]}}+\snorm{B_{I\times J^c}}\le \snorm{B_{I^c\times[n]}}+\snorm{B_{[n]\times J^c}} \le 2\snorm{B},\]
where the first inequality is the triangle inequality, and the remaining inequalities use that zeroing out a subset of rows and/or columns cannot increase the operator norm. 
\end{proof}

We also require the following bound on the operator norm of the matrix in terms of the $\ell_1$ norms of its rows and columns, due to Schur \cite{Sch11}.
\begin{lemma}\label{lem:schur-bound}
For any matrix $A$,  
\[\snorm{A}\le \snorm{A}_{\infty\to 1}^{1/2}\snorm{A^T}_{\infty\to 1}^{1/2}.\]
\end{lemma}


Now, we proceed to the treatment of the three classes of entries. The easiest is the third class, for which the result follows by a straightforward application of the Chernoff bound. 

\begin{lemma}[{Modification of \cite[Corollary~8.6]{RV18}}]\label{lem:large}
Let $\epsilon\in (0,1/2]$. Consider an $n\times n$ random matrix upper triangular matrix $T$ with independent (but not necessarily identically distributed) entries $T_{ij}$ such that $T_{ij} = 0$ or $|T_{ij}|\ge 5\sqrt{n/\epsilon}$, and which satisfy $\mb{E}T_{ij}^2\le 1$. Then with probability $1-\exp(-\epsilon n)$, all nonzero entries of $T$ are contained in an $\epsilon n \times \epsilon n$ matrix.
\end{lemma}
\begin{remark}
In \cite{RV18}, the above lemma is stated for identically distributed entries but as can be seen from the proof, this assumption is unnecessary.
\end{remark}

Next, we handle entries in the first class. The following lemma, and its proof (which we include for completeness), are essentially identical to \cite[Proposition~8.4]{RV18}. 
\begin{lemma}
\label{lem:medium-low}
Let $\epsilon \in (0,1/2]$. Let $T$ be a random upper triangular $n\times n$ matrix with i.i.d.~entries of mean $0$ and variance at most $1$ such that $T_{ij} = 0$ or $|T_{ij}|\in[\sqrt{n/\log\epsilon^{-1}},\sqrt{n/(\epsilon\log^2\epsilon^{-1})}]$ almost surely. With probability at least $1-2\exp(-\epsilon n/4)$ we can zero out an $\epsilon n\times\epsilon n$ block to obtain $\wt{T}$ satisfying
\[\snorm{\wt{T}}\le C_{\ref{lem:medium-low}}\sqrt{\frac{n}{\epsilon}}.\]
\end{lemma}

\begin{proof}
Let $B_{ij} = \mbm{1}_{T_{ij}\neq 0}$. Since the $B_{ij}$ are independent Bernoulli random variables with mean $p \le(\log\epsilon^{-1})/n$, it follows from \cite[Corollary~8.2]{RV18} that with probability at least $1-2\exp(-\epsilon n/4)$, we can find an $\epsilon n\times\epsilon n$ sub-matrix of $B$ to zero out so that the rows and columns of the resulting matrix have at most $O(\log\epsilon^{-1})$ ones each. Since $|T_{ij}|\le\sqrt{n/(\epsilon\log^2\epsilon^{-1})}$, this shows that we can find an $\epsilon n \times \epsilon n$ sub-matrix of $T$ to zero out so that the rows and columns of the resulting matrix have $\ell_1$ norm at most $O(\sqrt{n/\epsilon})$ each. An application of \cref{lem:schur-bound} gives the desired result.
\end{proof}

Finally, we handle entries in the second class. 

\begin{proposition}\label{prop:medium-high}
Let $\epsilon \in (0,1/2]$. Let $T$ be a random upper triangular $n\times n$ matrix with i.i.d.~entries of mean $0$ and variance at most $1$ such that $T_{ij} = 0$ or $|T_{ij}|\in[\sqrt{n/(\epsilon\log^2\epsilon^{-1})},5\sqrt{n/\epsilon}]$ almost surely. With probability at least $1-2\exp(-c_{\ref{prop:medium-high}}\epsilon n)$ we can zero out a square sub-matrix of size $C_{\ref{prop:medium-high}}\epsilon n \times C_{\ref{prop:medium-high}}\epsilon n$ to obtain $\wt{T}$ satisfying
\[\snorm{\wt{T}}\le C_{\ref{prop:medium-high}}\sqrt{\frac{n}{\epsilon}}.\]
\end{proposition}
\begin{proof}
As always, we can assume $n = \Omega(\epsilon^{-1})$ since otherwise, the target probability is negative. Consider the matrix $B_{ij} = \mbm{1}_{T_{ij}\neq 0}$. This is an upper triangular matrix with i.i.d.~$\on{Ber}(p)$ entries, where $p\le(\epsilon\log^2\epsilon^{-1})/n$. We will show that with probability $1-2\exp(-c\epsilon n)$, one can remove $O(\epsilon n)$ rows of $B$ to obtain $B'$ such that each row and each column of $B'$ has at most a single entry equal to $1$. Since $|T_{ij}| \le 5\sqrt{n/\epsilon}$, this shows that the corresponding matrix
$T'$ will satisfy $\snorm{T'}\le 5\sqrt{n/\epsilon}$. By using the symmetry $(i,j) \mapsto (n-j, n-i)$, it follows that with probability $1-2\exp(-c\epsilon n)$, one can remove $O(\epsilon n)$ columns of $T$ to obtain $T''$ with $\snorm{T''} \le 5\sqrt{n/\epsilon}$.  
Finally, \cref{lem:sub-matrix} shows that the matrix $\wt{T}$ obtained by removing the intersection of these $O(\epsilon n)$ rows and $O(\epsilon n)$ columns will have operator norm bounded by $O(\sqrt{n/\epsilon})$ as desired.

Therefore, consider $B$ as above. Let $t$ denote the number of rows of $B$ with exactly one entry equal to $1$ and let $b$ denote the number of rows of $B$ with at least two entries equal to $1$. Also, let $B^{(1)}$ denote the $t\times n$ sub-matrix of $B$ in which each row has exactly one entry equal to $1$. Then, by the Chernoff--Hoeffding bound \cite[Theorem~1]{Hoe63}, except with probability at most $2\exp(-c\epsilon n\log^2\epsilon^{-1})$, we have $t \le t^*:= \lceil 2\epsilon n \log^2\epsilon^{-1} \rceil$ and $b = O(\epsilon n)$. Hence, with probability at least $1-2\exp(-c\epsilon n\log^2\epsilon^{-1})$, we have enough room to remove all rows of $B$ with at least two entries equal to $1$. 

It remains to deal with $B^{(1)}$. Let $B^{(-)}$ denote the sub-matrix of $B^{(1)}$ consisting of the first $n_-$ columns and $B^{(+)}$ denote the sub-matrix of $B^{(1)}$ consisting of the remaining $n_+$ columns, where $n_-$ and $n_+$ are chosen to be as close to each other as possible. We wish to bound the probability
\[\mb{P}[B^{(\pm)} \in \mc{B} \mid t]\]
uniformly for all $t \le t^*$, where $\mc{B}$ denotes the collection of $\{0,1\}$-valued $t\times n_\pm$ matrices for which the number of non-zero entries present in the union of all columns with at least two non-zero entries is not $O(\epsilon n)$. We present the bound for $B^{(-)}$, noting that the same argument also applies to $B^{(+)}$. Clearly, the probability is maximized for $t = t^*$. Moreover, since the rows of $B^{(-)}$ are independent and since each row has the distribution
\[
R \sim
\begin{cases}
\mbf{0}, \text{ w.p. } n_-/n,\\
e_i, \text{ w.p. } n_+/(n_-\cdot n)\text{ for all }i\in [n_-],
\end{cases}
\]
it follows that for all $t \le t^*$ and $n \ge 10$,
\[\mb{P}[B^{(-)} \in \mc{B} \mid t] \le \mb{P}[B^{(-)} \in \mc{B} \mid t^*] \le \mb{P}[\wh{B} \in \mc{B}],\]
where $\wh{B}$ denotes a random $t^* \times n_-$ matrix whose entries are i.i.d.~$\on{Ber}(2/n)$ random variables. We note that the above trick of splitting into $B^{(\pm)}$ and passing to the independent model is  closely related to the proof of \cite[Theorem~4.1]{Coh16}.

To bound $\mb{P}[\wh{B} \in \mc{B}]$, we note two things. First, by the Chernoff-Hoeffding bound, with probability at least $1-2\exp(-c\epsilon n)$, there are at most $s^* = \epsilon n/\log\epsilon^{-1}$ columns of $\wh{B}$ with at least two non-zero entries. Second, assuming $s^* \ge 1$ (otherwise, we are already done) by the Chernoff-Hoeffding bound and the union bound, the probability that \emph{any} $t^* \times \lfloor s^* \rfloor$ block of $\wh{B}$ has more than $O(\epsilon n)$ entries is at most $2\exp(-c\epsilon n)$. Combining these two facts shows that $\mb{P}[\wh{B} \in \mc{B}] \le 4\exp(-c\epsilon n)$.

To summarize, we have shown that except with probability $O(\exp(-c\epsilon n))$, the following hold simultaneously: $b = O(\epsilon n)$, $B^{(-)} \notin \mc{B}$, $B^{(+)} \notin \mc{B}$. On this event, we are guaranteed that there are $O(\epsilon n)$ rows of $B$ which either have at least two non-zero entries, or which contain a non-zero entry in a column of $B^{(\pm)}$ with at least two non-zero entries. Then, we can simply zero out all such rows to obtain the desired conclusion. \qedhere

\end{proof}

\section{Proof of \texorpdfstring{\cref{thm:optimal-iid-norm-regularization,thm:sym-norm-regularization,thm:algorithmic-norm-regularization}}{Theorems 1.2 to 1.4}}
\label{sec:proofs}
We are now in position to prove all our results.

\begin{proof}[Proof of \cref{thm:optimal-iid-norm-regularization,thm:sym-norm-regularization,thm:algorithmic-norm-regularization}]
We may assume that $n = \Omega(\epsilon^{-1})$ since otherwise, the desired success probability is negative and the statements are vacuously true. Moreover, it suffices to prove a version of \cref{thm:algorithmic-norm-regularization} for i.i.d.~random upper triangular matrices since this readily implies \cref{thm:optimal-iid-norm-regularization,thm:sym-norm-regularization,thm:algorithmic-norm-regularization}.
We decompose
\[T = S + M_1 + M_2 + L,\]
where $S$ contains the entries with magnitude at most $\sqrt{n/\log\epsilon^{-1}}$, $M_1$ contains the entries with magnitude in $$I_1 = (\sqrt{n/\log\epsilon^{-1}},\sqrt{n/(\epsilon\log^2\epsilon^{-1})}],$$ $M_2$ contains the entries with magnitude in $$I_2 = (\sqrt{n/(\epsilon\log^2\epsilon^{-1})},5\sqrt{n/\epsilon}],$$ and $L$ contains the entries with magnitude greater than $5\sqrt{n/\epsilon}$.

Since $(M_1)_{ij} = T_{ij}\mbm{1}_{|T_{ij}| \in I_1}$, we see that $M_1$ satisfies the assumptions of \cref{lem:medium-low}. Therefore, with probability at least $1-2\exp(-\epsilon n/4)$, we can find an $\epsilon n \times \epsilon n$ sub-matrix of $M_1$ to zero out so that the resulting matrix $\snorm{\wt{M_1}}$ has norm at most $C_{\ref{lem:medium-low}}\sqrt{n/\epsilon}$. Moreover, the proof of \cite[Corollary~8.2]{RV18} shows that this step can be done algorithmically. Indeed, one only needs to find the rows and columns of $M_1$ with more than $O(\log{\epsilon^{-1}})$ non-zero entries and then zero out \emph{any} $\epsilon n \times \epsilon n$ sub-matrix containing all the entries contained in such rows and columns. This last step is easy since \cite[Corollary~8.2]{RV18} guarantees that there are at most $\epsilon n$ such entries so that they can be trivially placed in an $\epsilon n \times \epsilon n$ sub-matrix.   

Similarly, we can apply \cref{prop:medium-high} to $M_2$ in order to find a suitable $\epsilon n\times\epsilon n$ sub-matrices to remove. In order to do this algorithmically, let $R_1$ denote the set of rows with exactly one non-zero entry, let $R_2$ denote the set of rows with at least two non-zero entries, let $C_2(R_1)$ denote the set of columns with at least two non-zero entries in rows in $R_1$, and let $R_1'$ denote the set of rows in $R_1$ such that the unique non-zero entry is in a column in $C_2(R_1)$. Define $C_1, C_2, R_2(C_1), C_1'$ similarly with the role of rows and columns interchanged. Then, the proof of \cref{prop:medium-high} shows that, with probability at least $1-2\exp(-c_{\ref{prop:medium-high}}\epsilon n)$, the intersection of the rows $R_2 \cup R_1'$ and the columns $C_2 \cup C_1'$ is contained in an $C_{\ref{prop:medium-high}}\epsilon n \times C_{\ref{prop:medium-high}}\epsilon n$ sub-matrix and zeroing it out suffices to regularize the norm of $M_2$. 

To regularize the norm of $L$, we can apply \cref{lem:large}. Again, this is algorithmic, since we only need to remove a sub-matrix containing all non-zero entries (of which there are at most $\epsilon n$ with probability at least $1-\exp(-\epsilon n)$, see the proof of \cite[Corollary~8.6]{RV18}).


Finally, we regularize the norm of $S$. Let $\xi' = \xi \mbm{1}_{|\xi| \le \sqrt{n/\log\epsilon^{-1}}}$ where $\xi$ denotes the common distribution of the upper triangular entries of $T$ and note that the upper triangular entries of $S$ are i.i.d.~copies of $\xi'$. Since $\mb{E}[\xi'^{2}] \le \mb{E}[\xi^{2}] \le 1$, we can apply \cref{lem:sym-2->infty} to algorithmically find (with probability at least $1-2\exp(-c_{\ref{lem:sym-2->infty}}\epsilon n)$ a set $J_0\subseteq [n]$ of $O(\epsilon n)$ columns such that
\[\snorm{S_{J_0^{c}}}_{2\to \infty} = O(\sqrt{n}).\]
In order to be able to apply \cref{prop:restrictions}, we need such a result not for $S$ but for $S - \mb{E}S$. For this, we note that
\begin{align*}
|\mb{E}\xi'| &= \Big|\mb{E}\xi\mbm{1}_{|\xi|\le\sqrt{n/\log\epsilon^{-1}}}\Big| = \Big|\mb{E}\Big[\xi - \xi\mbm{1}_{|\xi| > \sqrt{n/\log\epsilon^{-1}}}\Big]\Big|\\
&= \Big|\mb{E}\xi\mbm{1}_{|\xi| > \sqrt{n/\log\epsilon^{-1}}}\Big|  
\le(\mb{E}\xi^2)^{1/2}\mb{P}\Big[|\xi| > \sqrt{n/\log\epsilon^{-1}}\Big]^{1/2}
\le\sqrt{\frac{\log\epsilon^{-1}}{n}}.
\end{align*}
Therefore, for all $J \subseteq [n]$,
\[\snorm{(\mb{E}S)_J}_{2 \to \infty} \le \sqrt{\log{\epsilon^{-1}}}, \quad \snorm{(\mb{E}S)_J} \le \sqrt{n\log{\epsilon^{-1}}}.\]
In particular,
\[\snorm{(S - \mb{E}S)_{J_0^{c}}}_{2 \to \infty} = O(\sqrt{n} + \sqrt{\log{\epsilon^{-1}}}) = O(\sqrt{n}).\]
Combining this with \cref{prop:restrictions}, we see that except with probability at least $1-4^{-n} - 2\exp(-c_{\ref{lem:sym-2->infty}}\epsilon n)$,
\begin{equation}
\label{eqn:recentered-infinity-2}
\snorm{(S-\mb{E}S)_{J_0^{c}}}_{\infty \to 2} = O(n).
\end{equation}
Therefore, by \cref{thm:grothendieck-pietsch}, we see that there exists some $\mu$ with $|\mu| \le \sqrt{\log{\epsilon^{-1}}/n}$ and a subset of columns $J_1$ with $J_1 = O(\epsilon n)$ such that for $J_* = J_1 \cup J_0$, 
\[\snorm{(S - \mu \mbf{1}\mbf{1}^{T})_{J_*^{c}}} = O(\sqrt{n/\epsilon}).\]
This immediately implies that
\[\snorm{S_{J_*^{c}}} \le O(\sqrt{n/\epsilon}) + O(\sqrt{n\log{\epsilon^{-1}}}) = O(\sqrt{n/\epsilon}).\]
In order to find $\mu$ and $J_1$ algorithmically, it suffices to apply a deterministic version of \cref{thm:grothendieck-pietsch} due to Tropp \cite[Theorem~3.1]{Tro09} to the matrices
\[(S - \frac{j}{\sqrt{n}}\mbf{1}\mbf{1}^{T})_{J_0^{c}}, \quad j = 0,\pm 1,\pm 2,\dots, \pm\lceil\sqrt{\log{\epsilon^{-1}}}\rceil.\]
By \cref{eqn:recentered-infinity-2}, \cref{thm:grothendieck-pietsch}, and the bound on $|\mb{E}\xi'|$, this procedure is guaranteed to succeed whenever \cref{eqn:recentered-infinity-2} holds. We note that alternatively, one may proceed by first learning the mean $\mb{E}\xi'$ to within additive error $1/\sqrt{n}$ and then using the algorithm of \cite[Theorem~3.1]{Tro09} on the matrix $(S - \wh{\mb{E}S})_{J_0^{c}}$, where $\wh{\mb{E}S}$ is our approximation for $\mb{E}S$. 



Next, by symmetry, one may repeat this procedure to find a collection of $O(\epsilon n)$ rows $K_*$ such that 
\[\snorm{S_{K_*^{c}}} = O(\sqrt{n/\epsilon}).\]
By \cref{lem:sub-matrix}, it follows that zeroing out the $O(\epsilon n) \times O(\epsilon n)$ sub-matrix formed by the intersection of the rows in $K_*$ and columns in $J_*$ makes the norm of the resulting matrix $O(\sqrt{n/\epsilon})$. We now use \cref{lem:sub-matrix} to combine the sub-matrices zeroed out for $S, M_1, M_2, L$ and rescale $\epsilon$ to complete the proof. 

For the running time of the algorithm, the algorithmic Grothendieck-Pietsch factorization in \cite{Tro09} can be performed in time $\tilde{O}(n^{7/2})$ and everything else can clearly be performed in time $\tilde{O}(n^{2})$. \qedhere



\end{proof}

\bibliographystyle{amsplain0.bst}
\bibliography{main.bib}

\providecommand{\bysame}{\leavevmode\hbox to3em{\hrulefill}\thinspace}
\providecommand{\MR}{\relax\ifhmode\unskip\space\fi MR }
\providecommand{\MRhref}[2]{%
  \href{http://www.ams.org/mathscinet-getitem?mr=#1}{#2}
}
\providecommand{\href}[2]{#2}
\begin{thebibliography}{10}

\bibitem{Coh16}
Michael~B. Cohen, \emph{Nearly tight oblivious subspace embeddings by trace
  inequalities}, Proceedings of the {T}wenty-{S}eventh {A}nnual {ACM}-{SIAM}
  {S}ymposium on {D}iscrete {A}lgorithms, ACM, New York, 2016, pp.~278--287.

\bibitem{FO05}
Uriel Feige and Eran Ofek, \emph{Spectral techniques applied to sparse random
  graphs}, Random Structures Algorithms \textbf{27} (2005), 251--275.

\bibitem{Hoe63}
Wassily Hoeffding, \emph{Probability inequalities for sums of bounded random
  variables}, J. Amer. Statist. Assoc. \textbf{58} (1963), 13--30.

\bibitem{LLV17}
Can~M. Le, Elizaveta Levina, and Roman Vershynin, \emph{Concentration and
  regularization of random graphs}, Random Structures Algorithms \textbf{51}
  (2017), 538--561.

\bibitem{LT91}
Michel Ledoux and Michel Talagrand, \emph{Probability in {B}anach spaces},
  Ergebnisse der Mathematik und ihrer Grenzgebiete (3) [Results in Mathematics
  and Related Areas (3)], vol.~23, Springer-Verlag, Berlin, 1991, Isoperimetry
  and processes.

\bibitem{Reb20}
Elizaveta Rebrova, \emph{Constructive regularization of the random matrix
  norm}, J. Theoret. Probab. \textbf{33} (2020), 1768--1790.

\bibitem{RV18}
Elizaveta Rebrova and Roman Vershynin, \emph{Norms of random matrices: local
  and global problems}, Adv. Math. \textbf{324} (2018), 40--83.

\bibitem{RV10}
Mark Rudelson and Roman Vershynin, \emph{Non-asymptotic theory of random
  matrices: extreme singular values}, Proceedings of the International Congress
  of Mathematicians 2010 (ICM 2010) (In 4 Volumes) Vol. I: Plenary Lectures and
  Ceremonies Vols. II--IV: Invited Lectures, World Scientific, 2010,
  pp.~1576--1602.

\bibitem{Sch11}
J.~Schur, \emph{Bemerkungen zur {T}heorie der beschr\"{a}nkten {B}ilinearformen
  mit unendlich vielen {V}er\"{a}nderlichen}, J. Reine Angew. Math.
  \textbf{140} (1911), 1--28.

\bibitem{Seg00}
Yoav Seginer, \emph{The expected norm of random matrices}, Combinatorics,
  Probability and Computing \textbf{9} (2000), 149--166.

\bibitem{Tro09}
Joel~A. Tropp, \emph{Column subset selection, matrix factorization, and
  eigenvalue optimization}, Proceedings of the {T}wentieth {A}nnual
  {ACM}-{SIAM} {S}ymposium on {D}iscrete {A}lgorithms, SIAM, Philadelphia, PA,
  2009, pp.~978--986.

\bibitem{BSY88}
Y.~Q. Yin, Z.~D. Bai, and P.~R. Krishnaiah, \emph{On the limit of the largest
  eigenvalue of the large-dimensional sample covariance matrix}, Probab. Theory
  Related Fields \textbf{78} (1988), 509--521.

\end{thebibliography}

\end{document}